\documentclass [a4paper,11pt]{article}

\usepackage[english]{babel}
\usepackage{amsmath}
\usepackage{amsfonts}
\usepackage{amssymb}
\usepackage{amsthm}								%enhanced theorems
\usepackage{enumerate}							%enhanced enumerations
\usepackage{setspace}
\usepackage{graphicx}
\usepackage{tikz}
\usepackage{hyperref}
\usepackage{latexsym}
\usepackage{verbatim}  % for multiline comments
	\newcommand*\savesymbol[1]{%
	  \expandafter\let\csname orig#1\expandafter\endcsname\csname#1\endcsname
	  \expandafter\let\csname #1\endcsname\relax
	}
	\newcommand*\restoresymbol[2]{%
	  \expandafter\global\expandafter\let\csname#1#2\expandafter\endcsname%
	    \csname#2\endcsname
	  \expandafter\global\expandafter\let\csname#2\expandafter\endcsname%
	    \csname orig#2\endcsname
	}

	\savesymbol{mathbb}

\usepackage{bbold}

	\restoresymbol{b}{mathbb}

\usepackage{fancyhdr}

\theoremstyle{plain}
	\newtheorem{theorem}{Theorem}[section]
	\newtheorem{corollary}[theorem]{Corollary}
	
	\newtheorem{lemma}[theorem]{Lemma}

	\newtheorem{tclaim}{Claim}[theorem]
	\newtheorem{tlemma}[tclaim]{Lemma}
\theoremstyle{definition}
	\newtheorem{definition}[theorem]{Definition}
	\newtheorem{example}[theorem]{Example}
	
\theoremstyle{remark}
	\newtheorem{remark}[theorem]{Remark}

\def\noqed{\renewcommand{\qedsymbol}{}}		% toglie qed a fine prova

\newcommand{\R}{\mathbb R}
\newcommand{\PP}{\mathbb P}
\newcommand{\QQ}{\mathbb Q}

\newcommand{\B}{\mathcal B}

\newcommand{\1}{\bmathbb 1}
\newcommand{\2}{\bmathbb 2}

\newcommand{\pp}{\mathcal P}

\newcommand{\res}{\upharpoonright}										% RESTRIZIONE DI FUNZIONI
\newcommand{\cp}[1]{\left( #1 \right)}
\newcommand{\qp}[1]{\left[ #1 \right]}

\newcommand{\ap}[1]{\langle #1 \rangle}
\newcommand{\bp}[1]{\left\lbrace #1 \right\rbrace}
\newcommand{\vp}[1]{\left\lvert #1 \right\rvert}

\DeclareMathOperator{\cf}{cf}
\DeclareMathOperator{\dom}{dom}
\DeclareMathOperator{\ran}{ran}
\DeclareMathOperator{\val}{val}

\DeclareMathOperator{\trcl}{trcl}

\DeclareMathOperator{\ZFC}{ZFC}
\DeclareMathOperator{\FA}{FA}
\DeclareMathOperator{\BFA}{BFA}
\DeclareMathOperator{\MA}{MA}
\DeclareMathOperator{\BMA}{BMA}
\DeclareMathOperator{\PFA}{PFA}
\DeclareMathOperator{\BPFA}{BPFA}
\DeclareMathOperator{\SPFA}{SPFA}
\DeclareMathOperator{\BSPFA}{BSPFA}
\DeclareMathOperator{\MM}{MM}
\DeclareMathOperator{\BMM}{BMM}
\DeclareMathOperator{\SRP}{SRP}
\DeclareMathOperator{\OCA}{OCA}
\DeclareMathOperator{\PID}{PID}
\DeclareMathOperator{\AD}{AD}
\DeclareMathOperator{\SCH}{SCH}
\DeclareMathOperator{\CF}{CF}
\DeclareMathOperator{\NS}{NS}

\newcommand{\ON}{\mathrm{\mathbf{ON}}}

\begin{document}

	\title{{\sc An introduction to forcing axioms,\\${\SRP}$ and ${\OCA}$}}
	\date{}
	\author{Giorgio Audrito, Gemma Carotenuto\\\small{\it notes based on lectures of Matteo Viale}}
	\maketitle

	\tableofcontents

%%%%%%%%%%%%%%%%%%%%%%%%%%%%%%%%%%%%%%%%%%%%%%%%%%%%%%%%%%%%%%%%%%%%%%%%%%%%%%%%%%%%%%%%%%%%%%%%%%%%%%%%

\section{Introduction}

These notes are extracted from the lectures on forcing axioms and applications held by professor Matteo Viale at the University of Turin in the academic year 2011-2012. Our purpose is to give a brief account on forcing axioms with a special focus on some consequences of them ($\SRP$, $\OCA$, $\PID$). These principles were first isolated by Todor\v cevi\'c in \cite{tod89} and interpolate most consequences of $\MM$ and $\PFA$, thus providing a useful insight on the combinatorial structure of the theory of forcing axioms.

In the first part of this notes we will give a brief account on forcing axioms (section \ref{sec:forcing}), introducing some equivalent definition by means of generalized stationarity (section \ref{sec:stationarity}), and presenting the consequences of them in terms of generic absoluteness.

In the second part (section \ref{sec:srp}) we will state the strong reflection principle ($\SRP$), prove it under $\MM$ and examine its main consequences. This axiom is defined in terms of reflection properties of generalized stationary set as introduced in \ref{sec:stationarity}.

In the third part (section \ref{sec:oca}) we will state the open coloring axiom ($\OCA$), and provide consistency proofs for some versions of it (sections \ref{sec:oca:zfc}, \ref{sec:oca:ad}). This axiom can be seen as a sort of two-dimensional perfect set property, i.e. the basic descriptive set theory result that every analytic set is either countable or it contains a perfect subset.

In the last part (section\ref{sec:gaps}) we will explore a notable application of $\OCA$ to problems concerning properties of the continuum, in particular the existence of certain kind of gaps in $\omega^\omega$.

%%%%%%%%%%%%%%%%%%%%%%%%%%%%%%%%%%%%%%%%%%%%%%%%%%%%%%%%%%%%%%%%%%%%%%%%%%%%%%%%%%%%%%%%%%%%%%%%%%%%%%%%

	\subsection{Notation}

		In this notes, $f[A]$ (resp. $f^{-1}[A]$) will denote the set $f[A] = \{f(x) : ~ x \in A\}$ (resp. with $f^{-1}$). We will use $\qp{X}^\kappa$ (resp. $\qp{X}^{<\kappa}$) to denote the set of all subsets of $X$ of size $\kappa$ (resp. less than $\kappa$).	$M_{\alpha}$ will be the stage $\alpha$ of the cumulative hierarchy in $M$, and $H(\kappa)$ will be the class of all sets hereditarily of cardinality $< \kappa$. We shall write $\phi^M$ to mean the interpretation of $\phi$ in the model $M$.

		If $M$ is a transitive model of $\ZFC$ with $\PP \in M$, $M^\PP$ will be the set of $\PP$-names in $M$, and $M[G]$ will be the forcing extension of $M$ with a filter $G$ that is $M$-generic for some $\PP$. We will use $\dot{A}$ to denote a $\PP$-name for $A \in M[G]$, $\check{A}$ to denote the standard $\PP$-name for $A \in M$, and $\val_G(\dot{A})$ to denote the evaluation of the $\PP$-name $\dot{A}$ with an $M$-generic filter $G$.

		We recall that given a poset $\PP$, a set $D \subseteq \PP$ is \emph{dense} iff for every $p \in \PP$ there exists a $q \in D$, $q \leq p$; and a filter $G$ is \emph{$M$-generic for $\PP$} iff $G \cap D \cap M \neq \emptyset$ for every $D \in M$ dense subset of $\PP$.

%%%%%%%%%%%%%%%%%%%%%%%%%%%%%%%%%%%%%%%%%%%%%%%%%%%%%%%%%%%%%%%%%%%%%%%%%%%%%%%%%%%%%%%%%%%%%%%%%%%%%%%%

%%%%%%%%%%%%%%%%%%%%%%%%%%%%%%%%%%%%%%%%%%%%%%%%%%%%%%%%%%%%%%%%%%%%%%%%%%%%%%%%%%%%%%%%%%%%%%%%%%%%%%%%

	\section{Generalized stationarity} \label{sec:stationarity}

		In this section we shall introduce a generalization of the notion of stationarity for subsets of cardinals to subsets of any set. This concept has been proved useful in many contexts, and is needed in our purpose to state the strong reflection principle $\SRP$. Reference texts for this section are \cite{Jech}, \cite[Chapter 2]{Larson}.

		\begin{definition}
			Let $X$ be an uncountable set. A set $C$ is a \emph{club} on $\pp(X)$ iff there is a function $f_C: ~ X^{<\omega} \rightarrow X$ such that $C$ is the set of elements of $\pp(X)$ closed under $f_C$, i.e.
			\[
			C = \bp{ Y \in \pp(X): ~ f_C[Y]^{<\omega} \subseteq Y }
			\]
			A set $S$ is \emph{stationary} on $\pp(X)$ iff it intersects every club on $\pp(X)$.
		\end{definition}

		\begin{example}
			The set $\bp{X}$ is always stationary since every club contains $X$. Also $\pp(X) \setminus \bp{X}$ and $\qp{X}^\kappa$ are stationary for any $\kappa \leq \vp{X}$ (following the proof of the well-known downwards L\"owhenheim-Skolem Theorem). Notice that every element of a club $C$ must contain $f_C(\emptyset)$, a fixed element of $X$.
		\end{example}

		\begin{remark}
			The reference to the support set $X$ for clubs or stationary sets may be omitted, since every set $S$ can be club or stationary only on $\bigcup S$.
		\end{remark}

		There is one more property of stationary sets that is worth to mention. Given any first-order structure $M$, from the set $M$ we can define a Skolem function $f_M: M^{<\omega} \rightarrow M$ (i.e., a function coding solutions for all existential first-order formulas over $M$). Then the set $C$ of all elementary submodels of $M$ contains a club (the one corresponding to $f_M$). Henceforth, every set $S$ stationary on $X$ must contain an elementary submodel of any first-order structure on $X$.

		\begin{definition}
			A set $S$ is \emph{subset modulo club} of $T$, in symbols $S \subseteq^* T$, iff $\bigcup S = \bigcup T = X$ and there is a club $C$ on $X$ such that $S \cap C \subseteq T \cap C$. Similarly, a set $S$ is \emph{equivalent modulo club} to $T$, in symbols $S =^* T$, iff $S \subseteq^* T \wedge T \subseteq^* S$.
		\end{definition}

		\begin{definition}
			The \emph{club filter} on $X$ is $\CF_X = \bp{C \subset \pp(X): ~ C \text{ contains a club} }$. Similarly, the \emph{non-stationary ideal} on $X$ is $\NS_X = \bp{A \subset \pp(X): ~ A \text{ not stationary}}$.
		\end{definition}

		\begin{remark}
			If $\vp{X} = \vp{Y}$, then $\pp(X)$ and $\pp(Y)$ are isomorphic and so are $CF_X$ and $CF_Y$ (or $NS_X$ and $NS_Y$): then we can suppose $X \in \ON$ or $X \supseteq \omega_1$ if needed.
		\end{remark}

		\begin{lemma}
			$\CF_X$ is a $\sigma$-complete filter on $\pp(X)$, and the stationary sets are exactly the $\CF_X$-positive sets.
		\end{lemma}
		\begin{proof}
			$\CF_X$ is closed under supersets by definition. Given a family of clubs $C_i$, $i < \omega$, let $f_i$ be the function corresponding to the club $C_i$. Let $\pi: \omega \rightarrow \omega^2$ be a surjection, with components $\pi_1$ and $\pi_2$, such that $\pi_2(n) \leq n$. Define $g: X^{<\omega} \rightarrow X$ to be $g(s) = f_{\pi_1(\vp{s})}(s \res \pi_2(\vp{s}))$. It is easy to verify that $C_g = \bigcap_{i < \omega} C_i$.
		\end{proof}

		\begin{definition}
			Given a family $\bp{S_a \subseteq \pp(X): ~ a \in X}$, the \emph{diagonal union} of the family is $\nabla_{a \in X} S_a = \bp{z \in \pp(X): ~ \exists a \in z ~ z \in S_a}$, and the \emph{diagonal intersection} of the family is $\Delta_{a \in X}  S_a = \bp{z \in \pp(X): \forall a \in z ~ z \in S_a}$.
		\end{definition}

		\begin{lemma}[Fodor] \label{sFodor}
			$\CF_X$ is normal, i.e. is closed under diagonal intersection. Equivalently, every function $f: ~ \pp(X) \rightarrow X$ that is regressive on a $\CF_X$-positive set is constant on a $\CF_X$-positive set.
		\end{lemma}
		\begin{proof}
			Given a family $C_a$, $a \in X$ of clubs, with corresponding functions $f_a$, let $g(a^\smallfrown s) = f_a(s)$. It is easy to verify that $C_g = \Delta_{a \in X} C_a$.

			Even though the second part of our thesis is provably equivalent to the first one for any filter $\mathcal F$, we shall opt here for a direct proof. Assume by contradiction that $f: ~ \pp(X) \rightarrow X$ is regressive (i.e., $f(Y) \in Y$) in a $\CF_X$-positive (i.e., stationary) set, and $f^{-1}\qp{a}$ is non-stationary for every $a \in X$. Then, for every $a \in X$ there is a function $g_a: ~ \qp{X}^{<\omega} \rightarrow X$ such that the club $C_{g_a}$ is disjoint from $f^{-1}\qp{a}$. Without loss of generality, suppose that $C_{g_a} \subseteq C_a = \bp{Y \subseteq X: ~ a \in Y}$. As in the first part of the lemma, define $g(a^\smallfrown s) = g_a(s)$. Then for every $Z \in C_g$ and every $a \in Z$, $Z$ is in $C_{g_a}$ hence is not in $f^{-1}\qp{a}$ (i.e., $f(Z) \neq a$). So $f(Z) \notin Z$ for any $Z \in C_g$, hence $C_g$ is a club disjoint with the stationary set in which $f$ is regressive, a contradiction.
		\end{proof}

		\begin{remark}
			The club filter is never $\omega_2$-complete, unlike its well-known counterpart on cardinals. Let $Y \subseteq X$ be such that $\vp{Y} = \omega_1$, and $C_a$ be the club corresponding to $f_a: \qp{X}^{<\omega} \rightarrow \bp{a}$; then $C = \bigcap_{a \in Y} C_a = \bp{Z \subseteq X: ~ Y \subseteq Z}$ is disjoint from the stationary set $\qp{X}^\omega$, hence is not a club.
		\end{remark}

		This generalized notion of club and stationary set is closely related to the well-known one defined for subsets of cardinals.

		\begin{lemma} \label{sClassicalOmg1}
			$C \subseteq \omega_1$ is a club in the classical sense if and only if $C \cup \bp{\omega_1}$ is a club in the generalized sense. $S \subseteq \omega_1$ is stationary in the classical sense if and only if it is stationary in the generalized sense.
		\end{lemma}
		\begin{proof}
			Let $C \subseteq \omega_1 + 1$ be a club in the generalized sense. Then $C$ is closed: given any $\alpha = \sup{\alpha_i}$ with $f[\alpha_i]^{<\omega} \subseteq \alpha_i$, $f[\alpha]^{<\omega} = \bigcup_i f[\alpha_i]^{<\omega} \subseteq \bigcup_i \alpha_i = \alpha$. Furthermore, $C$ is unbounded: given any $\beta_0 < \omega_1$, define a sequence $\beta_i$ by taking $\beta_{i+1} = \sup{f[\beta_i]^{<\omega}}$. Then $\beta_\omega = \sup{\beta_i} \in C$.

			Let now $C \subseteq \omega_1$ be a club in the classical sense. Let $C = \bp{c_\alpha: ~ \alpha < \omega_1}$ be an enumeration of the club. For every $\alpha < \omega_1$, let $\bp{d^\alpha_i: ~ i < \omega} \subseteq c_{\alpha+1}$ be a cofinal sequence in $c_{\alpha+1}$ (eventually constant), and let $\bp{e^\alpha_i: ~ i < \omega} \subseteq \alpha$ be an enumeration of $\alpha$. Define $f_C$ to be $f_C((c_\alpha)^n) = d^\alpha_n$, $f_C(0^\smallfrown \alpha^n) = e^\alpha_n$, and $f_C(s) = 0$ otherwise. The sequence $e^\alpha_i$ forces all closure points of $f_C$ to be ordinals, while the sequence $d^\alpha_i$ forces the ordinal closure points of $f_C$ being in $C$.
		\end{proof}

		\begin{lemma} \label{sClassicalK}
			If $\kappa$ is a cardinal with cofinality at least $\omega_1$, $C \subseteq \kappa$ contains a club in the classical sense if and only if $C \cup \bp{\kappa}$ contains the ordinals of a club in the generalized sense. $S \subseteq \kappa$ is stationary in the classical sense if and only if it is stationary in the generalized sense.
		\end{lemma}
		\begin{proof}
			If $C$ is a club in the generalized sense, then $C \cap \kappa$ is closed and unbounded by the same reasoning of Lemma \ref{sClassicalOmg1}. Let now $C$ be a club in the classical sense, and define $f: ~ \kappa^{< \omega} \rightarrow \kappa$ to be $f(s) = \min \bp{c \in C: \sup{s} < c}$. Then $C_f \cap \kappa$ is exactly the set of ordinals in $C \cup \bp{\kappa}$ that are limits within $C$.
		\end{proof}

		\begin{remark}
			If $S$ is stationary in the generalized sense on $\omega_1$, then $S \cap \omega_1$ is stationary (since $\omega_1+1$ is a club by Lemma \ref{sClassicalOmg1}), while this is not true for $\kappa > \omega_1$. In this case, $\pp(\kappa) \setminus (\kappa+1)$ is a stationary set: given any function $f$, the closure under $f$ of $\bp{\omega_1}$ is countable, hence not an ordinal.
		\end{remark}

		\begin{lemma}[Lifting and Projection] \label{sLifting}
			Let $X \subseteq Y$ be uncountable sets. If $S$ is stationary on $\pp(Y)$, then $S \downarrow X = \bp{B \cap X: ~ B \in S}$ is stationary. If $S$ is stationary on $\pp(X)$, then $S \uparrow Y = \bp{B \subseteq Y: ~ B \cap X \in S}$ is stationary.
		\end{lemma}
		\begin{proof}
			For the first part, given any function $f: ~ \qp{X}^{<\omega} \rightarrow X$, extend it in any way to a function $g: ~ \qp{Y}^{<\omega} \rightarrow Y$. Since $S$ is stationary, there exists a $B \in S$ closed under $g$, hence $B \cap X \in S \downarrow X$ is closed under $f$.

			For the second part, fix an element $x \in X$. Given any function $f: ~ \qp{Y}^{<\omega} \rightarrow Y$, replace it with a function $g: ~ \qp{Y}^{<\omega} \rightarrow Y$ such that for any $A \subset Y$, $g[A]$ contains $A \cup \bp{x}$ and is closed under $f$. To achieve this, fix a surjection $\pi: ~ \omega \rightarrow \omega^2$ (with projections $\pi_1$ and $\pi_2$) such that $\pi_2(n) \leq n$ for all $n$, and an enumeration $\ap{t^n_i: ~ i < \omega}$ of all first-order terms with $n$ variables, function symbols $f_i$ for $i \leq n$ (that represent an $i$-ary application of $f$) and a constant $x$. The function $g$ can now be defined as $g(s) = t^{\pi_2(\vp{s})}_{\pi_1(\vp{s})}(s \res \pi_2(\vp{s}))$. Finally, let $h: ~ \qp{X}^{<\omega} \rightarrow X$ be defined by $h(s) = g(s)$ if $g(s) \in X$, and $h(s) = x$ otherwise. Since $S$ is stationary, there exists a $B \in S$ with $h[B] \subseteq B$, but $h[B] = g[B] \cap X$ (since $x$ is always in $g[B]$) and $g[B] \supset B$, so actually $h[B] = g[B] \cap X = B \in S$. Then, $g[B] \in S \uparrow Y$ and $g[B]$ is closed under $f$ (by definition of $g$).
		\end{proof}

		\begin{remark}
			Following the same proof, a similar result holds for clubs. If $C_f$ is club on $\pp(X)$, then $C_f \uparrow Y = C_g$ where $g = f ~\cup~ \mathrm{Id}_{Y \setminus X}$. If $C_f$ is club on $\pp(Y)$ such that $\bigcap C_f$ intersects $X$ in $x$, and $g, h$ are defined as in the second part of Theorem \ref{sLifting}, $C_f \downarrow X = C_h$ is club. If $\bigcap C_f$ is disjoint from $X$, $C_f \downarrow X$ is not a club, but is still true that it contains a club (namely, $\cp{C_f \cap C_{\bp{x}}} \downarrow X$ for any $x \in X$).
		\end{remark}

		\begin{theorem}[Ulam] \label{sUlam}
			Let $\kappa$ be an infinite cardinal. Then for every stationary set $S \subseteq \kappa^+$, there exists a partition of $S$ into $\kappa^+$ many disjoint stationary sets.
		\end{theorem}
		\begin{proof}
			For every $\beta \in [\kappa,\kappa^+)$, fix a bijection $\pi_\beta: ~ \kappa \rightarrow \beta$. For $\xi < \kappa$, $\alpha < \kappa^+$, define $A^\xi_\alpha = \bp{\beta < \kappa^+: ~ \pi_\beta(\xi) = \alpha}$ (notice that $\beta > \alpha$ when $\alpha \in \ran(\pi_\beta)$). These sets can be fit in a $(\kappa \times \kappa^+)$-matrix, called \emph{Ulam Matrix}, where two sets in the same row or column are always disjoint. Moreover, every row is a partition of $\bigcup_{\alpha < \kappa^+} A^\xi_\alpha = \kappa^+$, and every column is a partition of $\bigcup_{\xi < \kappa} A^\xi_\alpha = \kappa^+ \setminus (\alpha+1)$.

			Let $S$ be a stationary subset of $\kappa^+$. For every $\alpha < \kappa^+$, define $f_\alpha: ~ S \setminus (\alpha+1) \rightarrow \kappa$ by $f_\alpha(\beta) = \xi$ if $\beta \in A^\xi_\alpha$. Since $\kappa^+ \setminus (\alpha+1)$ is a club, every $f_\alpha$ is regressive on a stationary set, then by Fodor's Lemma \ref{sFodor} there exists a $\xi_\alpha < \kappa$ such that $f^{-1}_\alpha\qp{\bp{\xi_\alpha}} = A^{\xi_\alpha}_\alpha \cap S$ is stationary. Define $g: ~ \kappa^+ \rightarrow \kappa$ by $g(\alpha) = \xi_\alpha$, $g$ is regressive on the stationary set $\kappa^+ \setminus \kappa$, again by Fodor's Lemma \ref{sFodor} let $\xi^* < \kappa$ be such that $g^{-1}\qp{\bp{\xi^*}} = T$ is stationary. Then, the row $\xi^*$ of the Ulam Matrix intersects $S$ in a stationary set for stationary many columns $T$. So $S$ can be partitioned into $S \cap A^{\xi^*}_\alpha$ for $\alpha \in T \setminus \bp{\min(T)}$, and $S \setminus \bigcup_{\alpha \in T \setminus \bp{\min(T)}} A^{\xi^*}_\alpha$.
		\end{proof}

		\begin{remark}
			In the proof of Theorem \ref{sUlam} we actually proved something more: the existence of a Ulam Matrix, i.e. a $\kappa \times \kappa^+$-matrix such that every stationary set $S \subseteq \kappa^+$ is compatible (i.e., has stationary intersection) with stationary many elements of a certain row.
		\end{remark}

%%%%%%%%%%%%%%%%%%%%%%%%%%%%%%%%%%%%%%%%%%%%%%%%%%%%%%%%%%%%%%%%%%%%%%%%%%%%%%%%%%%%%%%%%%%%%%%%%%%%%%%%

	\subsection{More on stationarity}

		In this section we present some notable definition and results about stationary sets that are not strictly needed for the rest of the notes. Reference text for this section is \cite[Chapter 2]{Larson}.

		\begin{definition}
			Let $X$ be an uncountable set, $\kappa < \vp{X}$ be a cardinal. A set $C$ is a \emph{club} on $\qp{X}^\kappa$ (resp. $\qp{X}^{<\kappa}$) iff there is a function $f_C: ~ X^{<\omega} \rightarrow X$ such that $C$ is the set of elements of $\qp{X}^\kappa$ (resp. $\qp{X}^{<\kappa}$) closed under $f_C$, i.e.
			\[
			C = \bp{ Y \in \qp{X}^\kappa: ~ f_C[Y]^{<\omega} \subseteq Y }
			\]
			A set $S$ is \emph{stationary} on $\qp{X}^\kappa$ (respectively $\qp{X}^{<\kappa}$) iff it intersects every club on $\qp{X}^\kappa$ (respectively $\qp{X}^{<\kappa}$).
		\end{definition}

		This definition is justified by the observation that $\qp{X}^\kappa$ (resp. $\qp{X}^{<\kappa}$) is stationary on $X$ for every $\kappa < \vp{X}$. As in the unrestricted case, the club sets on $\qp{X}^\kappa$ (resp. $\qp{X}^{<\kappa}$) form a normal $\sigma$-complete filter on $\qp{X}^\kappa$ (resp. $\qp{X}^{<\kappa}$). We can also state an analogous formulation of Lemma \ref{sLifting}, with additional care in the case $\qp{X}^\kappa$: in that case, the lifting $\qp{X}^\kappa \uparrow \qp{Y}^\kappa$ may not be a club on $\qp{Y}^\kappa$ if $\vp{X} < \vp{Y}$. For example, such a set is not a club if there exists a Completely J\'onsson cardinal above $\vp{Y}$ since its complement $\qp{Y}^\kappa \setminus \cp{\qp{X}^\kappa \uparrow \qp{Y}^\kappa} = \qp{X}^{<\kappa} \uparrow \qp{Y}^\kappa$ is stationary.
		
		\begin{lemma}[Lifting and Projection]
			Let $X \subseteq Y$ be uncountable sets, $\kappa < \vp{X}$ be a cardinal. If $C$ contains a club on $\qp{Y}^\kappa$ (resp. $\qp{Y}^{<\kappa}$), then $C \downarrow \qp{X}^\kappa = \cp{C \downarrow X} \cap \qp{X}^\kappa$ (resp. $C \downarrow \qp{X}^{<\kappa}$) contains a club on $\qp{X}^\kappa$ (resp. $\qp{X}^{<\kappa}$). If $C$ contains a club on $\qp{X}^{<\kappa}$, then $C \uparrow \qp{Y}^{<\kappa} = \cp{C \uparrow Y} \cap \qp{Y}^{<\kappa}$ contains a club on $\qp{Y}^{<\kappa}$.

			If $S$ is stationary on $\qp{Y}^{<\kappa}$, then $S \downarrow \qp{X}^{<\kappa}$ is stationary on $\qp{X}^{<\kappa}$. If $S$ is stationary on $\qp{X}^\kappa$ (resp. $\qp{X}^{<\kappa}$), then $S \uparrow \qp{Y}^\kappa$ is stationary on $\qp{Y}^\kappa$ (resp. with $\qp{Y}^{<\kappa}$).
		\end{lemma}

		We can now define a natural ordering on stationary sets, that can be used to define a poset of notable relevance in set theory.

		\begin{definition}
			Let $S$, $T$ be stationary sets. We write $S \leq T$ iff $\bigcup S \supseteq \bigcup T$ and $S \subseteq T \uparrow \cp{\bigcup S}$.
		\end{definition}

		\begin{definition}
			The \emph{full stationary tower} up to $\alpha$ is the poset $\PP_{<\alpha}$ of all the stationary sets $S \in V_\alpha$ ordered by $S \leq T$ as defined above. The stationary tower restricted to size $\kappa$ up to $\alpha$ is the poset $\QQ^\kappa_{<\alpha} = \bp{S \in V_\alpha: ~ S \subseteq \qp{\bigcup S}^\kappa \text{ stationary}}$ ordered by the same relation.
		\end{definition}

%%%%%%%%%%%%%%%%%%%%%%%%%%%%%%%%%%%%%%%%%%%%%%%%%%%%%%%%%%%%%%%%%%%%%%%%%%%%%%%%%%%%%%%%%%%%%%%%%%%%%%%%

%%%%%%%%%%%%%%%%%%%%%%%%%%%%%%%%%%%%%%%%%%%%%%%%%%%%%%%%%%%%%%%%%%%%%%%%%%%%%%%%%%%%%%%%%%%%%%%%%%%%%%%%

%%%%%%%%%%%%%%%%%%%%%%%%%%%%%%%%%%%%%%%%%%%%%%%%%%%%%%%%%%%%%%%%%%%%%%%%%%%%%%%%%%%%%%%%%%%%%%%%%%%%%%%%

	\section{Forcing axioms} \label{sec:forcing}

		Forcing is well-known as a versatile tool for proving consistency results. The purpose of forcing axioms is to turn it into a powerful tool for proving theorems: this intuition is partly justified by the following \emph{Cohen's Absoluteness Lemma} \ref{fCohen}.

		In the following notes we will use the notation $M \prec_n N$ to mean $M \prec_{\Sigma_n} N$ (or equivalently $M \prec_{\Pi_n} N$, $M \prec_{\Delta_{n+1}} N$). Reference text for this section is \cite[Chapter 3]{Bekkali}. We first recall the following lemma.
	
		\begin{lemma}[Levi's Absoluteness]
			Let $\kappa > \omega$ be a cardinal. Then $H(\kappa) \prec_1 V$.
		\end{lemma}
		\begin{proof}
			Given any $\Sigma_1$ formula $\phi = \exists x ~ \psi(x,p_1,\ldots,p_n)$ with parameters $p_1, \ldots, p_n$ in $H(\kappa)$, if $V \vDash \neg \phi$ also $H(\kappa) \vDash \neg \phi$ since $H(\kappa) \subseteq V$ and $\psi$ is $\Delta_0$ hence absolute for transitive models. Suppose now that $V \vDash \phi$, so there exists a $q$ such that $V \vDash \psi(q,p_1,\ldots,p_n)$. Let $\lambda$ be large enough so that $q \in H(\lambda)$.
			By downward L\"owenheim Skolem Theorem there exists an $M \prec H(\lambda)$ such that $q \in M$, $\trcl(p_i) \subseteq M$ for all $i < n$, and $\vp{M} = \omega \cup \vp{\bigcup_{i<n} \trcl(p_i)} < \kappa$. Let $N$ be the Mostowski Collapse of $M$, with $\pi: ~ M \rightarrow N$ corresponding isomorphism. Since $H(\lambda) \vDash \psi(q,p_1,\ldots,p_n)$, the same does $M$ and $N \vDash \psi(\pi(q),p_1,\ldots,p_n)$. Since $N$ is transitive of cardinality less than $\kappa$, $N \subseteq H(\kappa)$ so $\pi(q) \in H(\kappa)$ and $H(\kappa) \vDash \phi$.
		\end{proof}
	
		\begin{lemma}[Cohen's Absoluteness] \label{fCohen}
			Let $T$ be any theory extending $\ZFC$, and $\phi$ be any $\Sigma_1$ formula with a parameter $p$ such that $T \vdash p \subseteq \omega$. Then $T \vdash \phi(p)$ if and only if $T \vdash \exists \PP ~ (\1_\PP \Vdash \phi(p))$.
		\end{lemma}
		\begin{proof}
			The left to right implication is trivial (choosing a poset like $\PP = \2$). For the reverse implication, suppose that $V \vDash \exists \PP ~ (\1_\PP \Vdash \phi(\check p))$, let $\PP$ be any such poset and $\theta$ be such that $p, \PP \in V_\theta$ and $V_\theta$ satisfies a finite fragment of $T$ large enough to prove basic $\ZFC$ and $\1_\PP \Vdash \phi(p)$. Let $M$, $N$ be defined as in the previous lemma (considering $p$ as the parameter, $\PP$ as the variable), then $N \vDash \cp{\1_\QQ \Vdash \phi(p)}$ where $\QQ = \pi(\PP)$. Let $G$ be $N$-generic for $\QQ$, so that $N[G] \vDash \phi(p)$. Since $\phi$ is $\Sigma_1$, $\phi$ is upward absolute for transitive models, hence $V \vDash \phi(p)$. The thesis follows by completeness of first-order logic.
		\end{proof}

		Cohen's Absoluteness Lemma can be generalized to the case $p \subseteq \kappa$ for any cardinal $\kappa$. However, to achieve that we need the following definition.

		\begin{definition}
			We write $\FA_\kappa(\PP)$ as an abbreviation for the sentence ``for every $\mathcal D \subset \pp(\PP)$ family of open dense sets of $\PP$ with $\vp{\mathcal D} \leq \kappa$, there exists a filter $G\subset \PP$ such that $G \cap D \neq \emptyset$ for all $D \in \mathcal{D}$''.
		\end{definition}

		In an informal sense, assuming the forcing axiom for a broad class of posets suggests that a number of different forcing has already been done in our model of set theory. This intuitive insight is reflected into the following equivalence.

		\begin{theorem} \label{fFACharact}
			Let $\PP$ be a poset and $\theta > 2^{\vp{\PP}}$ be a cardinal. Then $\FA_{\kappa}(\PP)$ holds iff there exists an $M \prec H(\theta)$, $\vp{M} = \kappa$, $\PP \in M$, $\kappa \subset M$ and a $G$ filter $M$-generic for $\PP$.
		\end{theorem}
		\begin{proof}
			First, suppose that $\FA_{\kappa}(\PP)$ holds and let $M \prec H(\theta)$ be such that $\PP \in M$, $\kappa \subset M$, $\vp{M} = \kappa$. There are at most $\kappa$ dense subsets of $\PP$ in $M$, hence by $\FA_{\kappa}(\PP)$ there is a filter $G$ meeting all those sets. However, $G$ might not be $M$-generic since for some $D \in M$, the intersection $G \cap D$ might be disjoint from $M$. Define:
			\[
			 N = \bp{x \in H(\theta): ~ \exists \tau \in M \cap V^\PP ~ \exists q \in G ~ \cp{q \Vdash \tau = \check{x}}}
			 \]
			 Clearly, $N$ cointains $M$ (hence contains $\kappa$), and the cardinality $\vp{N} \leq \vp{M \cap V^\PP} = \kappa$ since every $\tau$ can be evaluated in an unique way by the elements of the filter $G$. To prove that $N \prec H(\theta)$, let $\exists x \phi(x,a_1,\ldots,a_n)$ be any formula with parameters $a_1,\ldots,a_n \in N$ which holds in $V$. Let $\tau_i \in M^\PP$, $q_i \in G$ be such that $q_i \Vdash \tau_i = \check{a_i}$ for all $i < n$. Define $Q_\phi = \bp{p \in \PP: ~ p \Vdash \exists x \in V ~ \phi(x, \tau_1, \ldots, \tau_n) }$, this set is definable in $M$ hence $Q_\phi \in M$. Furthermore, $Q_\phi \cap G$ is not empty since it contains any $q \in G$ below all $q_i$. By fullness in $H(\theta)$, we have that:
			 \[
			 \begin{array}{l}
				 H(\theta) \vDash \forall p \in Q_\phi ~ p \Vdash \exists x \in V ~ \phi(x, \tau_1, \ldots, \tau_n) \Rightarrow \\
				 H(\theta) \vDash \exists \tau ~ \forall p \in Q_\phi ~ p \Vdash \tau \in V \wedge \phi(\tau, \tau_1, \ldots, \tau_n) \Rightarrow \\
				 M \vDash \exists \tau ~ \forall p \in Q_\phi ~ p \Vdash \tau \in V \wedge \phi(\tau, \tau_1, \ldots, \tau_n)
			\end{array}
			 \]
			 Fix such a $\tau$, by elementarity the last formula holds also in $H(\theta)$ and in particular for $q \in Q_\phi$. Since the set $\bp{ p \in \PP: ~ \exists x \in H(\theta) ~ p \Vdash \check{x} = \tau }$ is an open dense set definable in $M$, there is a $q' \in G$ below $q$ belonging to this dense set, and an $a \in H(\theta)$ such that $q' \Vdash \tau = \check{a}$. Then $q'$, $\tau$ testify that $a \in N$ hence the original formula $\exists x \phi(x,a_1,\ldots,a_n)$ holds in $N$.

			 Finally, we need to check that $G$ is $N$-generic for $\PP$. Let $D \in N$ be a dense subset of $\PP$, and $\dot{D} \in M$ be such that $\1_\PP \Vdash \dot{D} \text{ is dense} \wedge \dot{D} \in V$ and for some $q \in G$, $q \Vdash \dot{D} = D$. Since $\1_\PP \Vdash \dot{D} \cap \dot{G} \neq \emptyset$, by fullness lemma there exists a $\tau \in H(\theta)$ such that $\1_\PP \Vdash \tau \in \dot{D} \cap \dot{G}$, and by elementarity there is such a $\tau$ also in $M$. Let $q' \in G$ below $q$ be deciding the value of $\tau$, $q' \Vdash \tau = \check{p}$. Since $q'$ forces that $\check{p} \in \dot{G}$, it must be $q' \leq p$ so that $p \in G$ hence $p \in G \cap D \cap N$ is not empty.

			For the converse implication, let $M$, $G$ be as in the hypothesis of the theorem, and fix a collection $\mathcal D = \ap{D_\alpha: \alpha < \kappa}$ of dense subsets of $\PP$. Define:
			\[
			S = \bp{N \prec H(\vp{\PP}^+): ~ \kappa \subset N ~ \wedge ~ \vp{N} = \kappa ~ \wedge ~ \exists G \text{ filter } N\text{-generic }}
			\]
			Note that $S$ is definable in $M$ then $S \in M$. Furthermore, since $\PP \in M$ so is $H(\vp{\PP}^+)$ hence $M \cap H(\vp{\PP}^+) \prec H(\vp{\PP}^+)$ and $M \cap H(\vp{\PP}^+)$ is in $S$. Given any $C_f \in M$ club on $H(\vp{\PP}^+)$, since $f \in M$ we have that $M \cap H(\vp{\PP}^+) \in C_f$. Then $V \vDash S \cap C_f \neq \emptyset$ and by elementarity the same holds for $M$. Thus, $S$ is stationary in $M$ and again by elementarity $S$ is stationary also in $V$.

			Let $N \in S$ be such that $\mathcal D \in N$. Since $\kappa \subset N$ and $\mathcal D$ has size $\kappa$, $D_\alpha \in N$ for every $\alpha < \kappa$. Thus, the $N$-generic filter $G$ will meet all dense sets in $\mathcal D$, verifying $\FA_\kappa(\PP)$ for this collection.
		\end{proof}

		\begin{corollary} \label{fFACharactCoroll}
			Let $\PP$ be a poset with $\pp(\PP) \in H(\theta)$. Then $\FA_{\kappa}(\PP)$ holds if and only if there are stationary many $M \prec H(\theta)$ such that $\vp{M} = \kappa$, $\PP \in M$, $\kappa \subset M$ and a $G$ filter $M$-generic for $\PP$.
		\end{corollary}
		\begin{proof}
			The forward implication has already been proved in the first part of the proof of the previous Theorem \ref{fFACharact}. The converse implication directly follows from the same theorem.
		\end{proof}

		\begin{lemma}[Generalized Cohen's Absoluteness]
			Let $T$ be any theory extending $\ZFC$, $\kappa$ be a cardinal, $\phi$ be a $\Sigma_1$ formula with a parameter $p$ such that $T \vdash p \subseteq \kappa$. Then $T \vdash \phi(p)$ if and only if $T \vdash \exists \PP ~ \cp{\1_\PP \Vdash \phi(p) ~\wedge~ \FA_\kappa(\PP)}$.
		\end{lemma}
		\begin{proof}
			The forward implication is trivial; the converse implication follows the proof of Lemma \ref{fCohen}. Given $p$, $\PP$ such that $\1_\PP \Vdash \phi(p)$ and $\FA_\kappa(\PP)$ holds, by Corollary \ref{fFACharactCoroll} let $M \prec H(\theta)$ be such that $\vp{M} = \kappa$, $\PP \in M$, $\kappa \subset M$ and there exists a $G$ filter $M$-generic for $\PP$. Since there are stationary many such $M$, we can assume that $p \in M$. Let $\pi: M \rightarrow N$ be the transitive collapse map of $M$, then $H = \pi[G]$ is $N$-generic for $\QQ = \pi[\PP]$ and $p \subseteq \kappa \subseteq M$ is not moved by $\pi$ so that $N[H] \vDash \phi(p)$. Since $\phi$ is $\Sigma_1$, $\phi$ is upward absolute for transitive models, hence $V \vDash \phi(p)$.
		\end{proof}

		It is now clear how the forcing axiom makes forcing a strong tool for proving theorems. For $\kappa = \omega_1$, the forcing axiom $\FA_{\omega_1}(\PP)$ is widely studied for many different poset $\PP$. In particular, for the classes of posets:
		\[
		\mathrm{c.c.c.} ~ \subset ~ \mathrm{proper} ~ \subset ~ \mathrm{semiproper} ~ \subset ~ \mathrm{locally ~~ s.s.p.}
		\]
		the forcing axiom is called respectively $\MA$ (Martin's Axiom), $\PFA$ (Proper Forcing Axiom), $\SPFA$ (Semiproper Forcing Axiom), $\MM$ (Martin's Maximum). In this notes we will be mostly interested in the latter.
	
		\begin{definition}
			A poset $\PP$ is c.c.c. iff every antichain in $\PP$ is countable.
		\end{definition}
	
		\begin{definition}
			A poset $\PP$ is proper iff for every $\theta$ regular cardinal such that $\pp(\PP) \in H(\theta)$, countable elementary substructure $M \prec H(\theta)$ and $p \in \PP \cap M$, there is a condition $q \leq p$ that is $M$-generic (i.e., for every $D \in M$ dense subset of $\PP$ and $r \leq q$, $r$ is compatible with an element of $D \cap M$).
		\end{definition}
		
		Equivalently, a poset $\PP$ is proper iff it preserves stationary sets on $\qp{\lambda}^\omega$ for any $\lambda$ uncountable cardinal.
	
		\begin{definition}
			A poset $\PP$ is semiproper iff for every $\theta$ regular cardinal such that $\pp(\PP) \in H(\theta)$, countable elementary substructure $M \prec H(\theta)$ and $p \in \PP \cap M$, there is a condition $q \leq p$ that is $M$-semigeneric (i.e., for every $\dot{\alpha} \in M$ name for a countable ordinal, $q \Vdash \exists \beta \in M ~ \check{\beta} = \dot{\alpha}$).
		\end{definition}

		Under $\SPFA$ every s.s.p.~poset is semiproper and viceversa, hence $\SPFA$ is equivalent to $\MM$.

		\begin{definition}
			A poset $\PP$ is \emph{stationary set preserving} (in short, s.s.p.) iff for every stationary set $S \subseteq \omega_1$, $\1_\PP \Vdash \forall x \subseteq \check{\omega_1} (x \mathrm{~club~} \Rightarrow x \cap \check S \neq \emptyset)$.
		\end{definition}

		\begin{definition}
			A poset $\PP$ is \emph{locally} s.s.p.~iff there exists a $p \in \PP$ such that $\PP \res p = \bp{q \in \PP: ~ q \leq p}$ is an s.s.p.~poset.
		\end{definition}

		The class of locally s.s.p.~posets play a special role in the development of forcing axioms: $\MM$ is the strongest possible form of forcing axiom for $\omega_1$. This is the case as shown by the following theorem.

		\begin{theorem}[Shelah]
			If $\PP$ is not locally s.s.p.~then $\FA_{\omega_1}(\PP)$ is false.
		\end{theorem}
		\begin{proof}
			Given $\PP$ that is not locally s.s.p.~let $S$ be a stationary set on $\omega_1$ and $\dot C \in V^\PP$ be such that $\1_\PP \Vdash \dot{C} \subseteq \check{\omega_1} \mathrm{~club}$, $\1_\PP \Vdash \check{S} \cap \dot{C} = \check{\emptyset}$. Define:
			\[
			\begin{array}{ccl}
				D_\alpha &=& \bp{p \in \PP: ~ p \Vdash \check{\alpha} \in \dot{C} \vee p \Vdash \check{\alpha} \notin \dot{C} } \\
				E_\beta &=& \bp{p \in \PP: ~ p \Vdash \check{\beta} \notin \dot{C} \Rightarrow \exists \gamma < \beta ~ p \Vdash \dot{C} \cap \check{\beta} \subseteq \check{\gamma} } \\
				F_\gamma &=& \bp{p \in \PP: ~ \exists \alpha > \gamma ~ p \Vdash \check{\alpha} \in \dot{C} }
			\end{array}
			\]
			Those sets are dense by the forcing theorem, since $\dot{C}$ is forced to be a club and the above formulas are true for clubs (hence forced by a dense set of conditions). Suppose by contradiction that $\FA_{\omega_1}(\PP)$ holds, and let $G$ be a filter that intersects all the $D_\alpha$, $E_\beta$, $F_\gamma$. Then the set $C = \bp{\alpha < \omega_1: ~ \exists p \in G ~ p \Vdash \alpha \in \dot{C} }$ is a club in $V$, so there is a $\beta \in S \cap C$. By definition of $C$, there exists a condition $q \in G$ such that $q \Vdash \beta \in \dot{C}$, and $\beta \in S \Rightarrow q \Vdash \beta \in \check{S} \cap \dot{C} \neq \check{\emptyset}$, a contradiction.
		\end{proof}

%%%%%%%%%%%%%%%%%%%%%%%%%%%%%%%%%%%%%%%%%%%%%%%%%%%%%%%%%%%%%%%%%%%%%%%%%%%%%%%%%%%%%%%%%%%%%%%%%%%%%%%%

	\subsection{More on forcing axioms}

		In this section we will state a few interesting results without proof, not directly involved in the development of $\MM$ and $\SRP$. Reference texts for this section are \cite{Viale}, \cite{Viale2}. Cohen's Absoluteness Lemma \ref{fCohen} is a valuable result, but is limiting in two aspects. First, it involves only $\Sigma_1$ formulas, and second, forces the parameter to be a subset of $\omega$ (or of larger cardinals, assuming stronger and stronger versions of forcing axioms). The following Woodin's Absoluteness Lemma, with an additional assumption on large cardinals, enhances Cohen's result to any formula relativized to $L(\R)$.

		\begin{theorem}[Woodin's Absoluteness]
			Let $T$ be a theory extending $\ZFC ~ +$ there are class many Woodin cardinals. Let $\phi$ be any formula with a parameter $p$ such that $T \vdash p \subseteq \omega$. Then $T \vdash \phi(p)^{L(\R)}$ if and only if $T \vdash \exists \PP ~ (\1_\PP \Vdash \phi(\check p)^{L(\R)})$.
		\end{theorem}

		We would expect to generalize Woodin's result from $L(\R) = L(\pp(\omega))$ to some bigger class by means of forcing axioms, as we did with Cohen's. This happens to be possible, at least for $L(\qp{\ON}^{<\omega_2})$, by a result of Viale.
		To state it we need to introduce some common variations of the forcing axiom.

		\begin{definition}
			We write $\BFA_\kappa(\B)$ as an abbreviation for the sentence ``for every $\mathcal D \subset \qp{\B}^{\leq \kappa}$ family of predense sets of $\B$ with $\vp{\mathcal D} \leq \kappa$, there exists a filter $G\subset \B$ such that $G \cap D \neq \emptyset$ for all $D \in \mathcal{D}$''. If $\PP$ is a poset, we write $\BFA_\kappa(\PP)$ to mean $\BFA_\kappa(\B)$ for $\B$ the regular open algebra of $\PP$.
		\end{definition}

		The bounded forcing axiom $\BFA_\kappa(\PP)$ can be used to define weaker versions of the usual forcing axioms: $\BMA$, $\BPFA$, $\BMM$. Furthermore, $\BFA_\kappa(\PP)$ has an interesting equivalent formulation in terms of elementary substructures: namely, $\BFA_\kappa(\PP)$ holds if and only if $H(\kappa^+) \prec_1 V^\PP$.

		\begin{definition}
			We write $\FA^{++}_{\omega_1}(\PP)$ as an abbreviation for the sentence ``for every $\mathcal D \subset \pp(\PP)$ family of open dense sets of $\PP$ with $\vp{\mathcal D} \leq \omega_1$, there exists a filter $G\subset \PP$ such that $G \cap D \neq \emptyset$ for all $D \in \mathcal{D}$ and $\val_G(\dot{S})$ is stationary for every $\dot{S} \in V^\PP$ such that $\1_\PP \Vdash \dot{S} \subseteq \omega_1 \text{ stationary}$''.
		\end{definition}

		The forcing axiom $\FA^{++}_{\omega_1}(\PP)$ can be used to define analogous versions of the usual forcing axioms: $\MA^{++}$, $\PFA^{++}$, $\MM^{++}$. It is also possible to find an equivalent formulation of $\FA^{++}_{\omega_1}(\PP)$ similar to Theorem \ref{fFACharact}.

		While $\MA^{++}$ is provably equivalent to $\MA$, $\MM^{++}$ is an actual strengthening of $\MM$. These axioms also have distinct consistency strengths: for example, $\BPFA$ and $\BSPFA^{++}$ are consistent relative to a reflecting cardinal, while $\BMM$ is consistent relative to $\omega$-many Woodin cardinals, and $\MM^{++}$ is consistent relative to a supercompact cardinal.

		\begin{theorem}
			Let $\PP$ be a poset with $\pp(\PP) \in H(\theta)$. Then $\FA^{++}_{\omega_1}(\PP)$ holds if and only if there exists an $M \prec H(\theta)$, $\vp{M} = \omega_1$, $\PP \in M$, $\omega_1 \subset M$ and a $G$ filter $M$-generic for $\PP$ such that for every $\dot{S} \in V^\PP \cap M$ name for a stationary subset of $\omega_1$, $\val_G(\dot{S})$ is stationary.
		\end{theorem}

		We are now ready to state the concluding results of this section, generalizations of Woodin's Absoluteness Lemma.

		\begin{theorem}[Viale]
			Let $T$ be a theory extending $\ZFC ~ + \MM^{++} + $ there are class many Woodin cardinals. Let $\phi$ be any $\Sigma_2$ formula with a parameter $p$ such that $T \vdash p \in H(\omega_2)$. Then $T \vdash \phi(p)^{H(\omega_2)}$ iff $T \vdash \exists \PP \in \mathrm{SSP} ~~ \1_\PP \Vdash \cp{\phi(\check p)^{H(\omega_2)} \wedge \BMM}$.
		\end{theorem}

		\begin{theorem}[Viale]
			Let $T$ be a theory extending $\ZFC ~ + \MM^{+++} + $ there are class many supercompact cardinals limit of supercompact cardinals. Let $\phi$ be any formula with a parameter $p$ such that $T \vdash p \in H(\omega_2)$. Then $T \vdash \phi(p)^{L(\qp{\ON}^{\omega_1})}$ if and only if $T \vdash \exists \PP \in \mathrm{SSP} ~~ \1_\PP \Vdash \cp{\phi(\check p)^{L(\qp{\ON}^{\omega_1})} \wedge \MM^{+++}}$.
		\end{theorem}

%%%%%%%%%%%%%%%%%%%%%%%%%%%%%%%%%%%%%%%%%%%%%%%%%%%%%%%%%%%%%%%%%%%%%%%%%%%%%%%%%%%%%%%%%%%%%%%%%%%%%%%%

%%%%%%%%%%%%%%%%%%%%%%%%%%%%%%%%%%%%%%%%%%%%%%%%%%%%%%%%%%%%%%%%%%%%%%%%%%%%%%%%%%%%%%%%%%%%%%%%%%%%%%%%

%%%%%%%%%%%%%%%%%%%%%%%%%%%%%%%%%%%%%%%%%%%%%%%%%%%%%%%%%%%%%%%%%%%%%%%%%%%%%%%%%%%%%%%%%%%%%%%%%%%%%%%%

	\section{Strong Reflection Principle} \label{sec:srp}

		In the study of the consequences of $\MM$, there are certain statements that have been proved useful in isolating many of the characteristics of $\MM$: among those, the most prominents are the strong reflection principle $\SRP$, the open coloring axiom $\OCA$, and the $P$-ideal dichotomy $\PID$. Reference text for this section is \cite[5A]{Bekkali}. In this section we shall state the first one, prove it under $\MM$ and examine its consequences. We first need the following definition.

		\begin{definition}
			A set $S \subseteq \qp{X}^\omega$ is \emph{projectively stationary} iff it is stationary, $\omega_1 \subseteq X$, and its restriction $S \downarrow \omega_1 = \bp{A \cap \omega_1: ~ A \in S}$ contains a club on $\qp{\omega_1}^\omega$.
		\end{definition}

		The property of being projectively stationary will be mostly used by means of the following lemma.

		\begin{lemma} \label{rPrjIntersection}
			Let $S \subseteq \qp{X}^\omega$ be projectively stationary, and $T \subset \omega_1$ be stationary. Then $S \cap \cp{T \uparrow X}$ is stationary.
		\end{lemma}
		\begin{proof}
			Given a club $C$ on $X$, $S' = S \cap C$ is clearly projectively stationary. Let $\alpha$ be in $T \cap \cp{S' \downarrow \omega_1}$, and $A \in S'$ such that $A \cap \omega_1 = \alpha$. Then $A \in S \cap \cp{T \uparrow X} \cap C$.
		\end{proof}

		\begin{definition}
			A stationary set $S \subseteq \pp(X)$ \emph{reflects on $Z$} iff $Z \subseteq X$ and $S \cap \pp(Z)$ is stationary (notice that $S \downarrow Z$ is necessarily stationary while $S \cap \pp(Z)$ may not). A stationary set $S \subseteq \pp(X)$ \emph{strongly reflects on $Z$} iff $S \cap \qp{Z}^\omega$ contains a club on $\qp{Z}^\omega$.
		\end{definition}

		\begin{definition}
			We call \emph{strong reflection principle} on $X$ and write $\SRP(X)$ as an abbreviation for the sentence ``every projectively stationary set on $\qp{X}^\omega$ strongly reflects on some $Z \supseteq \omega_1$ of size $\omega_1$''. We say \emph{strong reflection principle} (and write $\SRP$) to mean ``$\SRP(X)$ for all $X \supseteq \omega_1$''.
		\end{definition}

		The reflection property can be restated in the following equivalent way.

		\begin{lemma} \label{rPrjStatContF}
			$\SRP(X)$ holds iff for every projectively stationary $S \subset \qp{X}^{\omega}$ there exists a continuous increasing function $f: ~ \omega_1 \rightarrow S$ with $\bigcup \ran(f) \supseteq \omega_1$.
		\end{lemma}
		\begin{proof}
			First, suppose that $\SRP(X)$ holds and let $S \subset \qp{X}^{\omega}$ be a projectively stationary set. Let $Z \supset \omega_1$ be such that $S$ strongly reflects on $Z$. Fix an enumeration $\ap{z_\alpha: ~ \alpha < \omega_1}$ of $Z$, and let $Z_\alpha = \bp{z_\beta: ~ \beta < \alpha}$. The set $C_1 = \bp{Z_\alpha: ~ \alpha < \omega_1}$ is a club on $\qp{Z}^\omega$ (by a similar argument to the one for $\omega_1$ club in Lemma \ref{sClassicalOmg1}). Since $S$ strongly reflects on $Z$, $S \cap C_1 = \bp{Z_\alpha: ~ Z_\alpha \in S }$ contains a club $C_2$. Thus, the increasing enumeration of $C_2$ is a continuous increasing function $f: ~ \omega_1 \rightarrow S$ with $\bigcup \ran(f) = Z \supseteq \omega_1$, as required.

			Conversely, suppose there exists a function $f: \omega_1 \rightarrow S$ as above, and define $Z = \bigcup \ran(f)$. Then $S \cap \qp{Z}^\omega$ contains $\ran(f)$ that is a club on $\qp{Z}^\omega$ by the same argument as above.
		\end{proof}

		Notice that the requirement $\ran(f) \supseteq \omega_1$ prevents $f$ to be eventually constant. To prove that $\SRP$ is a consequence of $\MM$, we shall define a poset $\PP_S$ that forces a projectively stationary set $S$ to strongly reflect on some $Z \supseteq \omega_1$, and argue that this poset is s.s.p.~for any $S$.

		\begin{definition}
			Given $S$ a projectively stationary set, $\PP_S$ is the poset of all the continuous increasing functions $f: ~ \alpha+1 \rightarrow S$ with $\alpha < \omega_1$ ordered by reverse inclusion.
		\end{definition}

		\begin{lemma} \label{rDense}
			The following sets are open dense in $\PP_S$ for $\alpha < \omega_1$, $a \in \bigcup S$:
			\[
			\begin{array}{rcl}
				D_\alpha &=& \bp{f \in \PP_S: ~ \alpha \in \dom(f)} \\
				E_a &=& \bp{f \in \PP_S: ~ a \in \bigcup \ran(f)}
			\end{array}
			\]
		\end{lemma}
		\begin{proof}
			For the first part, given any $f \in \PP_S$, $f: ~ \beta+1 \rightarrow S$ define $g \in D_\alpha$ below $f$ to be constant after $\beta$, i.e. $g(\gamma) = f(\beta)$ for every $\gamma \in \alpha+1 \setminus \beta$, $g(\gamma) = f(\gamma)$ otherwise.

			For the second part, given any $f \in \PP_S$, $f: ~ \beta+1 \rightarrow S$ let $A$ be any set in the intersection of $S$ with the club $C_{f(\beta) \cup \bp{a}} = \bp{Y \subseteq X: ~ f(\beta) \cup \bp{a} \subseteq Y}$. Then $g = f \cup \ap{\beta+1, A} \in E_a$ extends $f$ and is in $E_a$.
		\end{proof}

		\begin{lemma} \label{rTodo2}
			$\PP_S$ is an s.s.p.~poset.
		\end{lemma}
		\begin{proof}
			Let $T \subseteq \omega_1$ be a stationary set, and $\dot C$ be a $\PP_S$-name for a club. Given any $p \in \PP_S$, we need to find a $q \leq p$, $\delta \in T$ such that $q \Vdash \check{\delta} \in \dot C$.

			Let $M$ be a countable elementary submodel of $H(\theta)$ such that $p, S, T, \dot{C} \in M$ and $M \cap \bigcup S \in S$, $M \cap \omega_1 = \delta \in T$ (such an $M$ exists by Lemma \ref{rPrjIntersection} and lifting). Fix an enumeration $\ap{A_n: ~ n < \omega}$ of the $\PP_S$-dense sets in $M$, and define a sequence $p_n$ such that $p_0 = p$, $p_{n+1} \in A_n$ and $p_{n+1} \leq p_n$. Then $p_\omega = \bigcup_{n < \omega} p_n$ is a function from $\delta$ to $S$, since $p_\omega$ is below all $D_\alpha$ as in Lemma \ref{rDense} for $\alpha \in M \cap \omega_1 = \delta$. Furthermore, $\bigcup p_\omega[\delta] = M \cap \bigcup S$, since $p_\omega$ is below all $E_a$ as in Lemma \ref{rDense} for $a \in M \cap \bigcup S$. Then $q = p_\omega \cup \ap{\delta, M \cap \bigcup S}$ is continuous, hence $q \in \PP_S$. Moreover, $q \Vdash \check{\delta} \in \dot C$: given any generic filter $G$ containing $q$, $G$ is generic also for $M$ hence $M[G] \vDash \val_G(\dot{C}) \text{ club on } \omega_1$, but $M[G] \cap \omega_1 = \delta$ so $\val_G(\dot{C}) \cap \delta$ is unbounded and $\delta \in \val_G(\dot{C})$. This holds for any $G \ni q$ hence $q \Vdash \check{\delta} \in \dot{C}$, $\delta \in T$.
		\end{proof}

		\begin{theorem}[Todorcevic]
			$\MM \Rightarrow \SRP$.
		\end{theorem}
		\begin{proof}
			Let $S$ be a projectively stationary set, and $\PP_S$ be defined as in Lemma \ref{rTodo2}. For every $\alpha < \omega_1$, $D_\alpha$, $E_\alpha$ are open dense sets by Lemma \ref{rDense}. From Lemma \ref{rTodo2} we know that $\PP_S$ is s.s.p., so using $\MM$ we get a filter $G$ meeting all $D_\alpha$, $E_\alpha$ for $\alpha < \omega_1$. Define $g = \bigcup G: ~ \omega_1 \rightarrow S$, then $g$ is a continuous increasing function with $\bigcup \ran(g) \supseteq \omega_1$ hence by Lemma \ref{rPrjStatContF}, $\SRP$ holds.
		\end{proof}

		The strong reflection principle has a number of interesting consequences. The most known is the following result on cardinal arithmetic.

		\begin{theorem} \label{rCardinalArithm}
			Assume $\SRP(\kappa)$ with $\kappa$ regular cardinal. Then $\kappa^{\omega_1} = \kappa^{\omega} = \kappa$.
		\end{theorem}
		\begin{proof}
			Let $\ap{E_\alpha: ~ \alpha < \kappa}$ be a partition of $\bp{\alpha \in \kappa: ~ \cf(\alpha) = \omega}$ into stationary sets by Ulam Theorem \ref{sUlam}. Similarly, let $\ap{D_\alpha: ~ \alpha < \omega_1}$ be a partition of $\omega_1 \setminus \bp{0}$ into stationary sets such that $\min D_\alpha > \alpha$. To accomplish this, from $\ap{B_\alpha: ~ \alpha < \omega_1}$ partition of $\omega_1$ into stationary sets define $A_\alpha = B_\alpha \setminus \alpha+1$, $A_0 = \cp{\omega_1 \setminus \bp{0}} \setminus \bigcup_{0 < \alpha < \omega_1} A_\alpha$). Given $f: \omega_1 \rightarrow \kappa$, define $S_f = \bp{X \in \qp{\kappa}^\omega: ~ \forall \alpha ~ X \cap \omega_1 \in D_\alpha \Leftrightarrow \sup(X) \in E_{f(\alpha)} }$.
			\noqed
		\end{proof}

		\begin{tlemma} \label{rClaimSf}
			$S_f$ is projectively stationary for any $f$.
		\end{tlemma}
		\begin{proof}[Proof of Lemma]
			Let $A \subseteq \omega_1$ be stationary, and $C_g$ be the club corresponding to the function $g: ~ \kappa^{<\omega} \rightarrow \kappa$. We shall define an $X \in S_f \cap C_g \cap (A \uparrow \kappa)$ that testifies the projective stationarity of $S_f$.
			Let $h: ~ A \setminus \bp{0} \rightarrow \omega_1$ be defined by $h(\alpha) = \beta$ iff $\alpha \in D_\beta$. Since $\min(D_\beta) > \beta$, $h$ is a regressive function on the stationary set $A \setminus \bp{0}$. By Fodor's Lemma \ref{sFodor} let $\gamma$ be such that $f^{-1}\qp{\bp{\gamma}} = A \cap D_\gamma$ is stationary.

			Let $\ap{M_\alpha: ~ \alpha < \kappa}$ be a continuous strictly increasing sequence of elementary substructures of $H(\theta)$ (for some large $\theta$) of size less than $\kappa$, such that $g \in M_0$, $M_\alpha \in M_{\alpha+1}$, $\alpha \subset M_{\alpha+1}$. Since $M_\alpha \cap \kappa$ is an ordinal in club many $\alpha < \kappa$, by restricting to a subsequence we can assume that $M_\alpha \cap \kappa$ is an ordinal for all $\alpha < \kappa$.

			Then $C_1 = \bp{M_\alpha \cap \kappa: ~ \alpha < \kappa}$ is a club subset of $\kappa$, so there is a $\delta \in E_{f(\gamma)} \cap C_1$, hence a structure $M_\xi$ such that $M_\xi \cap \kappa = \delta \in E_{f(\gamma)}$. Since $\delta$ is in $E_{f(\gamma)}$, $\cf(\delta) = \omega$ and we can define an increasing sequence $\ap{\delta_i: ~ i < \omega}$ converging to $\delta$.

			Let $\ap{N_\alpha: ~ \alpha < \omega_1}$ be defined by letting $N_\alpha \in C_g$ be the closure under $g$ of the set $\bp{\delta_i: ~ i < \omega} \cup \alpha$. Since this set is a subset of $M_\xi$ and $g$ is in $M_\xi$ (that is closed under $g$), for all $\alpha$ the set $N_\alpha$ is a subset of $M_\xi$ hence $\sup(N_\alpha) = M_\xi \cap \kappa = \delta \in E_{f(\gamma)}$. Furthermore, the set $C_2 = \bp{\alpha < \omega_1: ~ N_\alpha \cap \omega_1 = \alpha}$ is a club: closed by continuity of the sequence, and unbounded since given $\alpha_0$ we can define $\alpha_{i+1} = \sup(N_{\alpha_i} \cap \omega_1)$ so that $\alpha_\omega = \sup_{i < \omega}\alpha_i \in C_2$.

			Thus, there exists a $\beta$ in the intersection of $C_2$ with the stationary set $A \cap D_\gamma$. The corresponding $N_\beta$ will be such that $N_\beta \cap \omega_1 = \beta \in A \cap D_\gamma$, and $N_\beta \in C_g$, $\sup(N_\beta) = \delta \in E_{f(\gamma)}$. So $N_\beta$ is in $S_f$, completing the proof of Lemma \ref{rClaimSf}.
		\end{proof}

		\begin{tclaim} \label{rClaimDeltaF}
			Given $f,g: ~ \omega_1 \rightarrow \kappa$, if there exists $h_f: \omega_1 \rightarrow S_f$, $h_g: \omega_1 \rightarrow S_g$ continuous increasing functions such that $\bigcup \ran(h_f) \supseteq \omega_1$, $\bigcup \ran(h_g) \supseteq \omega_1$ and $\sup \cp{\bigcup \ran(h_f)} = \sup \cp{\bigcup \ran(h_g)}$, then $f = g$.
		\end{tclaim}
		\begin{proof}[Proof of Claim]
			Note that by Lemma \ref{rPrjStatContF} functions $h_f$, $h_g$ satisfying all but the last condition exist. Let $C_1 = \bp{\alpha < \omega_1: ~ h_f(\alpha) \cap \omega_1 = h_g(\alpha) \cap \omega_1 = \alpha }$ be a club.

			Define $\delta_f^\alpha = \sup\cp{h_f(\alpha)}$, $\delta = \sup_{\alpha < \omega_1} \delta_f^\alpha$. Given any $\alpha \in D_\xi \cap C_1$ (for some $\xi$), there exists a $\beta > \alpha$ with $\beta \in D_\zeta \cap C_1$ (for some $\zeta \neq \xi$), so by definition of $S_f$ we have that $\delta_f^\alpha \in E_{f(\xi)}$, $\delta_f^\beta \in E_{f(\zeta)}$ and $\delta_f^\alpha \neq \delta_f^\beta$ (since $E_{f(\xi)} \cap E_{f(\zeta)} = \emptyset$). Then, the sequence $\ap{\delta_f^\alpha: ~ \alpha < \omega_1}$ is continuously increasing and not eventually constant, so the limit $\delta$ has cofinality $\omega_1$ and the sequence $\ap{\delta_f^\alpha: ~ \alpha < \omega_1}$ is club on $\delta$.

			The same argument holds for $\ap{\delta_g^\alpha: ~ \alpha < \omega_1}$, $\delta = \sup_{\alpha < \omega_1} \delta_g^\alpha$ (by hypothesis) and $C_2 = \bp{\alpha < \omega_1: ~ \delta_f^\alpha = \delta_g^\alpha} \cap C_1$ is a club: closed by continuity, unbounded since given any $\alpha_0 < \omega_1$ we can define $\alpha_{2i+1} = \min \bp{\beta \in C_1: ~ \delta_f^\beta \geq \delta_g^{\alpha_{2i}}}$, and $\alpha_{2i+2} = \min \bp{\beta \in C_1: ~ \delta_g^\beta \geq \delta_f^{\alpha_{2i+1}}}$, so that $\alpha_\omega = \sup_{i < \omega} \alpha_i$ is in $C_2$.

			Suppose by contradiction that $f \neq g$, and let $\beta$ be such that $f(\beta) \neq g(\beta)$, and $\gamma \in C_2 \cap D_\beta$. Then $f(\gamma) \cap \omega_1 = \gamma \in D_\beta$, $f(\gamma) \in S_f$ implies that $\delta^\gamma_f \in E_{f(\beta)}$. The same argument for $g$ implies that $\delta^\gamma_g \in E_{g(\beta)}$, but $\delta^\gamma_f = \delta^\gamma_g$ and $E_{f(\beta)}$ is disjoint from $E_{g(\beta)}$, a contradiction.
		\end{proof}

		\begin{proof}[Proof of Theorem \ref{rCardinalArithm}]
			Define a map $\pi: ~ {}^{\omega_1} \kappa \rightarrow \kappa$ to be $\pi(f) = \delta$ for $\delta$ least such that $\delta = \sup \cp{\bigcup \ran(h_f)}$ for some continuous increasing $h_f: ~ \omega_1 \rightarrow S_f$. By Claim \ref{rClaimDeltaF}, $\pi$ is well-defined and injective so $\vp{\kappa} \geq \vp{{}^{\omega_1} \kappa}$ hence $\kappa^{\omega_1} = \kappa$.
		\end{proof}

		\begin{corollary}
			$\MM \Rightarrow 2^{\aleph_0} = 2^{\aleph_1} = \aleph_2$.
		\end{corollary}
		\begin{proof}
			Since $\MM$ implies $\MA_{\omega_1}$, we know that $2^{\aleph_0} \geq \aleph_2$. But $\MM$ also implies $\SRP(\omega_2)$, then $2^{\aleph_0} \leq \aleph_2^{\aleph_0} = \aleph_2$. Similarly, $2^{\aleph_1} \leq \aleph_2^{\aleph_1} = \aleph_2$ hence $2^{\aleph_1} = \aleph_2$.
		\end{proof}

		\begin{remark}
			The purpose of cardinal arithmetic is to determine the value of $\lambda^\kappa$. Assuming $\MM$ we can determine the result at least for $\kappa \leq \aleph_2$ with $\kappa$ regular: in this case, $\lambda^\kappa = \max(\lambda,\aleph_2)$. Unfortunately, the consequences of $\MM$ in cardinal arithmetic for regular cardinals stop there (for example, the value of $\aleph_0^{\aleph_2}$ can be changed by forcing). However, $\MM$ implies the singular cardinal hypothesis $\SCH$. Our proof actually shows that assuming SRP $\lambda^\kappa = \lambda^+ + 2^\kappa$ for all $\lambda \geq \kappa \geq \cf(\lambda)$.
		\end{remark}

		The following corollary gives us an interesting example of projectively stationary set.

		\begin{corollary}
			Let $S$ be a stationary set on $\kappa$ restricted to cofinality $\omega$. Then $E(S) = \bp{X \in \qp{\kappa}^\omega: ~ \sup(X) \in S}$ is projectively stationary.
		\end{corollary}
		\begin{proof}
			The proof mimics the one of Lemma \ref{rClaimSf}. Let $A$, $C_g$, $\ap{M_\alpha: ~ \alpha < \kappa}$, $C'$ be defined as in the lemma above. Since $C'$ is a club, we can find a $\delta \in S \cap C'$, hence a structure $M_\xi$ such that $M_\xi \cap \kappa = \delta \in S$ so that $\cf(\delta) = \omega$.
			Let $\ap{\delta_i: ~ i < \omega}$, $\ap{N_\alpha: ~ \alpha < \omega_1}$, $C''$ be defined as in Lemma \ref{rClaimSf}. Recall that for all $\alpha$ the set $N_\alpha$ is a subset of $M_\xi$ in $C_g$ such that $\sup(N_\alpha) = M_\xi \cap \kappa = \delta \in S$ (i.e., $N_\alpha \in E(S)$). Since $C''$ is club, let $\beta$ be in $C'' \cap S$: the corresponding $N_\beta$ is in $E(S) \cap C_g \cap \cp{A \uparrow \kappa}$.
		\end{proof}

		The last consequence of $\SRP$ that we shall examine is the following Theorem \ref{rSaturation} about the structure of $\NS_{\omega_1}$.

		\begin{definition}
			An ideal $I$ on $\kappa$ is \emph{saturated} iff $\pp(\kappa)/I$ is a $\kappa^+$-c.c. poset.
		\end{definition}

		\begin{theorem} \label{rSaturation}
			$\SRP(\omega_2) \Rightarrow \NS_{\omega_1}$ saturated.
		\end{theorem}
		\begin{proof}
			$\SRP(\omega_2)$ implies that $\omega_2^{\omega_1} = \omega_2$ hence also $\vp{\pp(\omega_1)} = 2^{\omega_1} = \omega_2$, so that $\NS_{\omega_1}$ is necessarily $\omega_3$-cc. Assume by contradiction that $\NS_{\omega_1}$ is not saturated, then there exists a maximal antichain $\mathcal A = \ap{A_\alpha: ~ \alpha < \omega_2}$ in $\pp(\omega_1)/\NS_{\omega_1}$. Define $S = \bp{X \in \qp{\omega_2}^\omega: ~ \exists \delta \in X ~ X \cap \omega_1 \in A_\delta}$. We claim that $S$ is projectively stationary.

			Given any stationary $T \subseteq \omega_1$, and $g: ~ \omega_2^{<\omega} \rightarrow \omega_2$ with corresponding club $C_g$, we need to find an $X \in S \cap C_g$ (to prove the stationarity) such that $X \cap \omega_1 \in T$ (to prove the projective stationarity). By maximality of $\mathcal A$, let $\alpha < \omega_2$ be such that $T$ is compatible with $A_\alpha$ (i.e., $T \cap A_\alpha$ is stationary). Let $\ap{M_\beta: ~ \beta < \omega_1}$ be a continuous strictly increasing sequence of countable elementary substructures of $H(\omega_3)$ such that $\mathcal A, T, \alpha, g \in M_0$ and $\beta \in M_{\beta+1}$. Then $C = \bp{\beta < \omega_1: M_\beta \cap \omega_1 = \beta }$ is a club: closed by continuity of $\ap{M_\beta: ~ \beta < \omega_1}$, unbounded since for any $\beta_0$ in $\omega_1$ if $\beta_{i+1} = \sup(M_{\beta_i} \cap \omega_1)$, then $M_{\beta_\omega} \cap \omega_1 = \beta_\omega$ for $\beta_\omega = \sup_{i < \omega} \beta_i$. Let $\xi$ be in $T \cap A_\alpha \cap C$, then $M_\xi \in C_g$ since $g \in M_\xi$. Furthermore, $M_\xi \cap \omega_2 \in S$ since $M_\xi \cap \omega_1 = \xi \in A_\alpha \cap T$ (this proves also the projectivity) and $\alpha \in M_\xi$. This completes the proof that $S$ is projectively stationary.

			Since $S$ is projectively stationary on $\omega_2$ and $\SRP(\omega_2)$ holds, there is a $Z \supseteq \omega_1$ of size $\omega_1$ such that $S \cap \qp{Z}^\omega$ is club. Let $\beta$ be in $\omega_2 \setminus Z$, and define $T = S \cap \cp{A_\beta \uparrow Z}$ stationary set on $Z$. Let $g: ~ T \rightarrow Z$ be defined by $g(X) = \delta$ for a $\delta$ as in the definition of $S$ (i.e., such that $X \cap \omega_1 \in A_\delta$ and $\delta \in X$). The function $g$ is regressive on the stationary set $T$, then by Fodor's Lemma \ref{sFodor} there exists a fixed $\gamma \in Z$ (hence $\gamma \neq \beta$) such that $T' = g^{-1}\qp{\gamma}$ is a stationary subset of $T$. Since $T' = \bp{X \in \qp{Z}^\omega: ~ \gamma \in X ~\wedge~ X \cap \omega_1 \in A_\gamma \cap A_\beta}$, $T' \downarrow \omega_1$ is a stationary subset of $A_\gamma \cap A_\beta$, contradicting that $\mathcal A$ is an antichain.
		\end{proof}

%%%%%%%%%%%%%%%%%%%%%%%%%%%%%%%%%%%%%%%%%%%%%%%%%%%%%%%%%%%%%%%%%%%%%%%%%%%%%%%%%%%%%%%%%%%%%%%%%%%%%%%%

%%%%%%%%%%%%%%%%%%%%%%%%%%%%%%%%%%%%%%%%%%%%%%%%%%%%%%%%%%%%%%%%%%%%%%%%%%%%%%%%%%%%%%%%%%%%%%%%%%%%%%%%

\pagebreak
\section{Open Coloring Axiom} \label{sec:oca}

	This section and the following are currently under revision, and will be made available again soon.

\subsection{Formulations of open coloring principles in $\ZFC$} \label{sec:oca:zfc}
\subsection{Backgrounds on open colorings of a separable metric space} \label{sec:oca:backgrounds}
\subsection{Consistency of $\OCA_P$ under $\AD$} \label{sec:oca:ad}
\section{Applications of $\OCA$ to gaps in $\omega^\omega$} \label{sec:gaps}
\pagebreak

%\input{oca}

%%%%%%%%%%%%%%%%%%%%%%%%%%%%%%%%%%%%%%%%%%%%%%%%%%%%%%%%%%%%%%%%%%%%%%%%%%%%%%%%%%%%%%%%%%%%%%%%%%%%%%%%

%\input{gaps}

%%%%%%%%%%%%%%%%%%%%%%%%%%%%%%%%%%%%%%%%%%%%%%%%%%%%%%%%%%%%%%%%%%%%%%%%%%%%%%%%%%%%%%%%%%%%%%%%%%%%%%%%


\begin{thebibliography}{14}
	\addcontentsline{toc}{section}{References}

	\bibitem{ARS80} Abraham, Uri; Rubin, Matatyahu; Shelah, Saharon: On
  the consistency of some partition theorems for continuous colorings, 
  and the structure of $\aleph\sb 1$-dense real order types. {\it Ann. Pure Appl. Logic 29 (1985), no. 2, 123--206.}

	\bibitem{bag00} Bagaria, Joan: Bounded forcing axioms as principles of
  generic absoluteness. {\it  Arch. Math. Logic  39  (2000),  
  no. 6, 393--401} 

	\bibitem{bau80} Baumgartner, James E.: Applications of the proper
  forcing axiom. {\it Handbook of set-theoretic topology, 913--959, North-Holland, Amsterdam, 1984.}

	\bibitem{Bekkali} Bekkali, Mohamed:
				{Topics in Set Theory: Lebesgue Measurability, Large Cardinals, Forcing Axioms, Rho-functions.}
				\emph{Springer-Verlag, 1991.}

	\bibitem{dev78} Devlin, Keith J.: The Yorkshireman's guide to proper
  forcing. {\it Surveys in set theory, 60--115, London
  Math. Soc. Lecture Note Ser., 87, Cambridge Univ. Press, Cambridge,
  1983.}

	\bibitem{Jech}  Jech, Thomas: Set theory. The third millennium
  edition, revised and expanded. 
  {\it Springer Monographs in Mathematics. Springer-Verlag, Berlin, 2003. xiv+769 pp.} ISBN: 3-540-44085-2

	\bibitem{kechris}     Kechris, Alexander S.: Classical Descriptive Set Theory. {\it Springer Graduate Texts in Mathematics. Springer-Verlag, Berlin, 1995.} ISBN: 0-387-94374-9

	\bibitem{kun80}  Kunen, Kenneth: Set theory. An introduction to
  independence proofs. Reprint of the 1980 original. 
  {\it Studies in Logic and the Foundations of Mathematics, 102. North-Holland Publishing Co., Amsterdam, 1983.}

	\bibitem{Larson} Larson, Paul B.:
				{The Stationary Tower: Notes on a Course Given by W. Hugh Woodin.}
				\emph{American Mathematical Society, 2004.}

	\bibitem{moo02}  Moore, Justin Tatch: Open colorings, the continuum
  and the second uncountable cardinal.  
{\it Proc. Amer. Math. Soc.  130  (2002),  no. 9, 2753--2759 (electronic).}

	\bibitem{tod89} Todor\v cevi\'c, Stevo: Partition problems in topology.
{\it American Mathematical Society, Providence, RI, 1989.} 

	\bibitem{Viale} Viale, Matteo:
				{The category of stationary set preserving partial orders and $\MM^{++}$.}
				\emph{To be submitted, http://www2.dm.unito.it/paginepersonali/viale/.}

	\bibitem{Viale2} Viale, Matteo:
				{Martin's Maximum Revisited.}
				\emph{Preprint, arXiv:1110.1181, 2011.}


\end{thebibliography}
\end{document}